\documentclass[a4paper,oneside,12pt]{amsart}

\usepackage{graphicx}

\reversemarginpar

\theoremstyle{plain}

\newtheorem{thm}{Theorem}[section]
\newtheorem*{thmA}{Theorem A}

\newtheorem{prop}[thm]{Proposition}
\theoremstyle{definition}

\newtheorem*{Question}{\sf Question}

\numberwithin{equation}{section}
\def\cB{{\mathcal B}}
\def\cD{{\mathcal D}}
\def\cO{{\mathcal O}}
\def\bR{\mathbb{R}}
\def\cDA{{\cD(A)}}

\def\bC{{\mathbb C}}
\def\bD{{\mathbb D}}
\def\pt{\partial}
\newcommand\id{\operatorname{id}}
\newcommand\defin {\overset {\text {\rm def} }{=}}

\def\eps{\varepsilon}
\def\om{\omega}
\def\Re{\operatorname{Re}}

\def\Hol{\operatorname{Hol}}

  \newtheorem*{main thm}{Main Theorem}
 \newtheorem{proposition}[thm]{Proposition}

\def\beq{\begin{eqnarray}}
\def\eeq{\end{eqnarray}}
\def\beqa{\begin{eqnarray*}}
\def\eeqa{\end{eqnarray*}}

\def\pt{\partial}
\def\eps{\epsilon}

\def\beqn{\begin{equation}}
\def\eeqn{\end{equation}}

\def\cH{{\mathcal H}}

\def\cL{{\mathcal L}}

\def\mg#1{}

\renewcommand{\epsilon}{\varepsilon}
\renewcommand{\phi}{\varphi}

\begin{document}
%%%%%%%%%%%%%%%%%%%%%%%%%%%%%%%%%%%%%%%%%%%%%%%%%%%%%%%%%%%%%%%%%%%%%%%%
\title[On generators of $C_0$ semigroups]{On generators of $C_0$-semigroups\\ of composition operators}
\author{Eva A. Gallardo-Guti\'errez}
%\vskip-1cm
\address{E. A. Gallardo-Guti\'errez \newline Departamento de An\'alisis Matem\'atico,\newline
Facultad de Matem\'aticas,
\newline Universidad Complutense de
Madrid, \newline
 Plaza de Ciencias N$^{\underbar{\Tiny o}}$ 3, 28040 Madrid,  Spain
 \newline
and Instituto de Ciencias Matem\'aticas (CSIC-UAM-UC3M-UCM),
\newline Madrid,  Spain } \email{eva.gallardo@mat.ucm.es}
\thanks{First author is partially supported by Plan Nacional  I+D grant no. MTM2016-77710-P, Spain}
\author{Dmitry Yakubovich}
\address{D. V. Yakubovich\newline
Departamento de Matem\'aticas,\newline  Universidad Aut\'onoma de Madrid,\newline
Cantoblanco, 28049 Madrid, Spain\newline
and Instituto de Ciencias Matem\'aticas (CSIC-UAM-UC3M-UCM), \newline Madrid, Spain.}
\email {dmitry.yakubovich@uam.es}
\thanks{Second author is partially supported by Plan Nacional  I+D grant no. MTM2015-66157-C2-1-P,  the ICMAT Severo Ochoa
project SEV-2015-0554 of the Ministry of Economy and
Competitiveness of Spain, and by the European Regional Development
Fund}

\subjclass[2010]{47B35 (primary)}

\keywords{$C_0$-semigroup of composition operators}

% \date{February 2017}

\begin{abstract}
Avicou, Chalendar and Partington proved in \cite{AChP1} that an (unbounded) operator
$(Af)=G\cdot f'$ on the classical Hardy space generates a $C_0$ semigroup
of composition operators if and only if
it generates a quasicontractive semigroup.
Here we prove that if such an operator $A$ generates a $C_0$ semigroup, then it is automatically a semigroup
of composition operators, so that the condition of
quasicontractivity of the semigroup in the cited result is not necessary.
Our result applies to a rather general class of Banach spaces
of analytic functions in the unit disc.
\end{abstract}

\maketitle

%%%%%%%%%%%%%%%%%%%%%%%%%%%%%%%%%%%%%%%%%%%%%%%%%%%%%%%%%%%%%%%%%%%%%%%%%%%

\section{Introduction}

Let $\cB$ denote a Banach space. We recall that a one parameter family
$\{T_t\}_{t\ge 0}$ of bounded linear operators acting on $\cB$ is called a
semigroup if $T_0=I$ and
$T_tT_s=T_{t+s}$ for all $t,s\ge 0$.
It is called a $C_0$-semigroup if it is strongly continuous, that is,
$
\lim_{t\to 0^+} T_t f=f$ for any $f\in \cB$.

%
%  OLD:
%
%Given a $C_0$-semigroup $\{T_t\}_{t\geq0}$, there exists a closed
%and densely defined linear operator $A$ that determines the
%semigroup uniquely, called the generator of $\{T_t\}_{t\geq0}$,
%defined by

We recall that, given a $C_0$-semigroup $\{T_t\}_{t\geq0}$, its
generator $A$ is defined by
\[
Af=\lim_{t\to 0^+} \frac {T_t f-f} t
\]
for $f\in \cDA=\{x\in\cB: \lim_{t\to 0^+} \frac {T_t f-f} t \enspace \text{exists}\}$.
It is a closed and densely defined linear operator on $\cB$, and it determines the
semigroup uniquely.
Observe, as a consequence of the uniform boundedness theorem, that if
$\{T_t\}_{t\ge 0}$ is a $C_0$-semigroup on $\cB$, then there
exists $\om\in \bR$ and $M\ge 1$ such that
\begin{equation}\label{eq 1}
\|T_t\|\le M e^{\om t} \qquad \text{for all}\; t\ge 0,
\end{equation}
(see \cite[Chapter II]{EngNag-book} or \cite[Chapter 3]{ABHN}, for
instance). A semigroup satisfying (\ref{eq 1}) with $M=1$ is
called quasicontractive.

%By a semigroup of composition operators we mean a semigroup
%$T_t f= f\circ \phi_t$, where $\{\phi_t\}_{t\ge 0}$ is a so-called holomorphic flow on
%
In 1978, Berkson and Porta \cite{Ber-Por} gave a complete description of the generator $A$
of semigroups of composition operators acting on the classical Hardy space $H^2(\mathbb{D})$
induced by a \textit{holomorphic flow} of analytic self-maps of the unit disc $\mathbb{D}=\{z\in\bC: |z|<1\}$. Recall that
a \textit{holomorphic flow} in the open unit disc $\mathbb{D}$ is,  by definition
(see \cite{Sh}), a continuous family $\{\phi_t\}_{t\ge 0}$ of analytic self-mappings
of $\bD$ that has a semigroup property with respect to composition.
More precisely, a holomorphic flow has to meet the following conditions:

\begin{enumerate}
\item[1)]
$\phi_0(z)=z$, $\forall z\in \bD$;

\item[2)]
$\phi_{t+s}(z)=\phi_{t}\circ \phi_{s}(z)$, $\forall t,s\ge0, \, \forall \in \bD$;

%  OLD
%
%\item[3)]
%The mapping $(t,z)\mapsto \phi_{t}(z)$ is continuous
%on $[0,+\infty)\times \bD$.

\item[3)] For any $s\ge0$ and any $z\in\bD$,
$\lim_{t\to s}\phi_{t}(z)=\phi_{s}(z)$.
\end{enumerate}

The holomorphic flow
$\{\phi_t\}$ gives rise to a semigroup $T_t f= f\circ \phi_t$ of linear operators on $H^2(\mathbb{D})$, which is called a
semigroup of composition operators.
%
%Moreover,
%
Berkson and Porta \cite{Ber-Por} noticed that this semigroup
%
%induced by a holomorphic flow in $\mathbb{D}$
%
is always strongly continuous on
the Hardy space $H^2(\mathbb{D})$;  the same statement was proven by Siskakis regarding
other classical spaces of analytic functions, such as the Dirichlet
space $\mathcal{D}$. We refer to the survey \cite{Sis}, at this regards.

A straightforward computation shows that, at least for the case of $H^2(\mathbb{D})$ or of
$\mathcal{D}$, the generator $A$ of a semigroup of
composition operators is of the form $Af = Gf^{\prime}$, where $G$ is an analytic function in $\mathbb{D}$. Indeed,
as Berkson and Porta showed, $G$ is the infinitesimal generator of the holomorphic flow $\{\phi_t\}$, defined by means of the equation
$$
\frac{\partial \phi_t(z)}{\partial t}=G(\phi_t(z)),\qquad \mbox{for } t\in \mathbb{R}_+ \mbox{ and } z\in \mathbb{D}.
$$

Very recently, Avicou, Chalendar and Partington \cite{AChP1} have
provided a complete description of quasicontractive
$C_0$-semigroups of bounded operators acting either on the Hardy
space $H^2(\mathbb{D})$ or the Dirichlet space $\mathcal{D}$,
whose generator $A$ is of the form $Af = Gf^{\prime}$, where $G$
is an analytic function in $\mathbb{D}$. Indeed, their Theorems
3.9 and 4.1 in \cite{AChP1} include, in particular, the following
statement:

\begin{thmA}
Let $\cB$ be either the Hardy space $H^2(\mathbb{D})$ or the Dirichlet space $\cD$ and let
$Af= G\cdot f'$ for $f\in \cDA=\{f\in \cB: G\cdot f'\in \cB\}$. Then
$A$ generates a $C_0$-semigroup of composition operators on $\cB$ if and only if
$A$ generates a quasicontractive $C_0$-semigroup on $\cB$.
\end{thmA}

In fact, this assertion was stated in \cite{AChP1} under an additional assumption
$G\in \cB$, however, as the authors observe in \cite[p. 549]{AChP2}, the same proofs
work without this assumption. We refer to \cite{AChP1,AChP2} for more details, in particular,
to several characterizations of possible functions $G$ that may appear here.

The question whether there may exist an operator $Af= G\cdot f'$, which is a generator
of a $C_0$-semigroup (on $H^2(\bD)$ or on $\cD$), which does not consist of composition operators, remained open.

Our main result in this paper gives an answer to this question not only in the context of the Hardy space or the Dirichlet space, but also
for more general function spaces. As a particular instance of our main theorem, we will show that an
operator $Af= G\cdot f'$,  is a generator
of a $C_0$-semigroup (on $H^2(\bD)$ or on $\cD$) if and only if $A$ generates
on these spaces a $C_0$-semigroup of composition operators.

The paper is organized as follows. In
Section~\ref{Section 2}, after some preliminaries, we prove our
main result. Section~\ref{Section 3} contains some discussions,
and we also raise some questions there.

\medskip

\textbf{Remark.}
After this work had been submitted for publication,
W.~Arendt and I.~Chalendar informed us about their work in progress,
where they address the same question and obtain another versions
of our main result, which applies to a class of domains
in $\bC$. Their conditions on the functional
space are different from ours.

\section{$C_0$-semigroups on spaces of analytic functions}
\label{Section 2}

In what follows, $\cB$ will be a Banach space of holomorphic functions on the unit disc $\mathbb{D}$.
The space of bounded linear operators on $\cB$ will be denoted by $\cL(\cB)$.
We denote by $\Hol(\bD)$ the space of all holomorphic functions on $\bD$ and by
$\mathcal{O}(\overline \bD)$ the set of all functions, holomorphic
on (a neighborhood of) the closed unit disc $\overline \bD$. Both are
given a structure of linear topological vector spaces in a usual way. We impose the following
natural assumption on $\cB$:

\medskip
\begin{enumerate}

\item[$(\star)$] $\cO(\overline \bD)\hookrightarrow{} \cB \hookrightarrow{} \Hol(\bD)$, and both embeddings are continuous.

\end{enumerate}

\medskip

%We remark that, if $\{\phi_t\}$ is a holomorphic flow on $\bD$, $\cB$ is reflexive and the linear operators $T_t f = f\circ
%\phi_t$ are bounded on $\cB$ for $t\ge 0$, then they form a $C_0$ semigroup.
%
We remark that, if $\{\phi_t\}$ is a holomorphic flow on $\bD$ and the linear operators $$T_t f = f\circ
\phi_t$$ are bounded on $\cB$ for $t\ge 0$, then they form a semigroup of linear operators.
In this case, we will say that it is a semigroup
of composition operators. In fact, it is a $C_0$-semigroup whenever
$\cB$ is reflexive (see Section \ref{Section 3} below).

Under the hypothesis $(\star)$, the generator of $\{T_t\}_{t\ge0}$ has the form
\[
Af= G\cdot f',
\]
and one has $\cDA=\{f\in \cB: G\cdot f'\in \cB\}$, (see \cite[Theorem 2]{BlContrDMMPapa}, for instance).

% where $G$ is an analytic function on $\bD$.
\medskip

Our main theorem reads as follows:

\medskip

\begin{main thm} \label{thm-main}
Let $\cB$ be a Banach space of analytic functions on $\bD$
satisfying hypothesis $(\star)$.
%
%that  $\cO(\overline \bD)\hookrightarrow{} \cB \hookrightarrow{} \Hol(\bD)$, and both embeddings are continuous.
%
Let $G$ be an analytic function in $\bD$ and let $A$ be given
$Af= G\cdot f'$ for $f\in \cDA=\{f\in \cB: G\cdot f'\in \cB\}$.
Then $A$ generates a $C_0$-semigroup on $\cB$ if and only if $A$ generates a $C_0$-semigroup
of composition operators.
\end{main thm}

\medskip

In particular, by Theorem~A, for the cases of the Hardy space
$H^2(\bD)$ and the Dirichlet space $\cD$, this semigroup will be
necessarily quasicontractive. (In Section~\ref{Section 3}, we will
give more general statements.)

\begin{proof}[Proof of Main Theorem]
Assume $A$ generates a $C_0$-semigroup $\{T_t\}$ on $\cB$. Our goal is to show
the existence of a holomorphic flow $\{\phi_t\}_{t\ge 0}$ such that
$T_t f=f\circ \phi_t$ for all functions $f\in \cB$. Fix a radius $r\in (0,1)$ and
consider the Cauchy problem
$$
\eqno{(CP)}\qquad\qquad
\begin{cases}
\frac {\pt \phi_t(z)}{\pt t}=G(\phi_t(z)) \\
\phi_0(z)=z\qquad\qquad\qquad \big(z\in D(0,r)=\{z\in \bD: \; |z|<r\}\big).
\end{cases}
$$
The standard theory of ordinary differential equations in complex domain implies that
there exists $t_0>0$ and an analytic solution $\{\phi_t(z)\}$ of (CP), defined for
$z\in D(0,r)$ and
all complex $t$, $|t|<t_0$. See, for instance,~\cite{Hille-book}, Theorems~2.2.1 and~2.8.2
(the Cauchy-Kovalevskaya Theorem is also applicable here). Moreover, this solution is unique in the class of
smooth functions.

We will only need real times $t\in (-t_0,t_0)$. Since the differential equation in (CP) is autonomous, we
have the semigroup property:
$\phi_{t+s}(z)=\phi_{t}\circ\phi_{s}(z)$ whenever
$t,s,t+s\in(-t_0,t_0)$ and $z,\phi_{s}(z)\in D(0,r)$.
Indeed, by (CP), for fixed $s$ and $z$, both functions
$\phi_{t+s}(z)$ and $\phi_{t}(\phi_{s}(z))$
satisfy the same differential equation for $t\in (-t_0+|s|,t_0-|s|)$ and have the
same initial value at $t=0$.

Let us also assume that
there is some $r'\in (0,r)$ such that
$\phi_{t}(z)\in D(0,r)$ whenever
$t\in (-t_0,t_0)$ and $z\in D(0,r')$; this is achieved
by substituting $t_0$ with a smaller number. Therefore
\beqn
\label{phit}
\phi_{-t}\circ\phi_{t}(z)=z\qquad \text{if $t\in (-t_0,t_0)$ and $z\in D(0,r')$.}
\eeqn

In what follows, we denote $g'(z)=\frac {\pt g(z)}{\pt z}$.

Our first goal is to prove the following:

\

\noindent \underline{Claim 1:}
For any $f\in\cDA$,
\beqn
\label{star}
% \eqno{(*)}
\qquad
T_t f(z)= f\circ \phi_t(z),
    \qquad z\in D(0,r), \, 0\le t< t_0.
\eeqn

%
%  Claim 1 -- > Claim 2
%
%There exists some $t_1=t_1(f) \in (0,t_0)$ such that $f_t(z)=\tilde f_t(z)$ for all $z\in D(0,r)$ and $0\le t<t_1$.

% One has $\tilde f_t(z)= f\circ \phi_t(z)$ for all $z\in D(0,r)$ and $0\le t<t_0$.

\begin{proof}% [Proof of Claim 1]
Fix $f\in \cDA$ and denote
%
% \marginpar{cambiar $\tilde f_t \to f_t$}
%
\[
f _t(z)=T_t f(z), \quad z\in\bD, \; t\ge0.
\]
Then $f _t\in \cDA$ for all $t\ge 0$, and \beqn
\label{1}  % 1
\frac {\pt f _t(z)}{\pt t} = Af _t(z) = G(z)f _t'(z), \quad
z\in\bD, \; t\ge0. \eeqn By (CP), $\frac {\pt \phi_{-t}(z)}{\pt
t}=-G(\phi_{-t}(z))$. Calculating the derivative of
$f_t\circ\phi_{-t}(z)$ with respect to $t$, using \eqref{1} and
the chain rule we get
\[
\frac \pt {\pt t} \big(f _t\circ \phi_{-t}(z)\big)  =
       G(\phi_{-t}(z)) f '_t (\phi_{-t}(z))
                  - f '_t (\phi_{-t}(z)) G(\phi_{-t}(z))=0, \quad 0\le t<t_0.
\]
Therefore for any $z\in D(0,r)$ and any $t\in (0,t_0)$, $f _t\circ
\phi_{-t}(z)=f_0\circ \phi_0(z)=f(z)$. By~\eqref{phit},
$f_t(z)=f\circ \phi_t(z)$ for $z\in D(0, r')$, which by
analyticity of both sides on $D(0, r)$ implies Claim~1.
\end{proof}

Let us denote by $E_z$ the evaluation functional $E_zf \defin
f(z)$, which is continuous on $\cB$ for any $z\in \bD$ by
hypotheses. We can write down \eqref{star} as $$E_z(T_t f) =
E_{\phi(z)}f.$$ Since $\cDA$ is dense in $\cB$, a density argument
gives that \eqref{star} holds for any $f\in \cB$. Now, by applying
\eqref{star} to the identity function $\id$, $\id(z)\equiv z$,
we get $(T_t\id)(z)=\id\circ \phi_t (z)=\phi_t (z)$
(first for $|z|<r$ and then for all $z\in\bD$, by
analytic continuation). This implies that $\phi_t\in \cB$ for any
$t$, $0\le t< t_0$. In the same way, \eqref{star} also gives $(T_t
z^n)(a)=\phi_t (a)^n$ for all $a\in\bD$.

\medskip

\noindent \underline{Claim 2:} There is a positive $t_1 \le t_0$ such that
$|\phi_t(z)|< 1$ for all $z\in\bD$ and all $0\le t< t_1$.

\begin{proof}% [Proof of Claim 2]
Notice that $(\star)$ implies that
$\overline{\lim}_{n\to\infty}\big(\|z^n\|_\cB\big)^{1/n}\le 1$.
Fix some $\eps>0$. Then there exist some constants $C_\eps$ and
$N$ such that for any $t\in[0,t_0)$, $a\in \bD$ and any integer
$n\ge N$,
\[
|\phi_t(a)|^n=|T_t(z^n)(a)|\le \|E_a\|\|T_t\|\|z^n\|_\cB\le C_\eps \|E_a\|\|T_t\| (1+\eps)^n.
\]
Taking $n$th roots, letting $n\to\infty$ and then $\eps\to 0$, we get that $|\phi_t(a)|\le 1$ for any $a\in\bD$.

Since $\phi_t(0)$ depends continuously on $t$ and $\phi_0(0)=0$, there is some $t_1\in(0,t_0]$
such that $|\phi_t(0)|< 1$ for
$0\le t< t_1$. It follows that $|\phi_t(z)|<1$ for  $0\le t< t_1$ and all $z\in\bD$, which completes
the proof.
%  of Claim 2.
\end{proof}

\medskip

It was shown above that the functions $\phi_t=T_t\id$, $0\le
t<t_1$, satisfy $\phi_s\circ\phi_t(z)=\phi_{s+t}(z)$ for $s, t\ge
0$, $s+t<t_1$ and $z\in D(0,r)$. Obviously, this equality extends
to all $z\in\bD$. The function $\phi_t(z)=E_z (T_t\id)$ is
continuous in $t\in [0,t_1)$. By \cite{Sh}, Proposition 3.3.1, the
family $\{\phi_t(z)\}$ can be continued to a holomorphic flow,
defined on $[0,+\infty)\times \bD$.

Finally, given any $t>0$, fix some $N$
such that $t/N<t_1$. Then, for any $f\in \cB$ it follows
\[
T_t f= T_{t/N}^N f= f\circ \underbrace{\phi_{t/N}\circ\cdots \circ \phi_{t/N}}_{\text{$N$ times}}
= f\circ \phi_t,
\]
which concludes the proof of the Main Theorem.
\end{proof}

In a recent paper \cite{ChalPart2016},
Chalendar and Partington prove some
analogues of the results of \cite{AChP1}, \cite{AChP2}
for generators given by higher order differential expressions
in the disc. The corresponding operator semigroups in general do not have
such clear geometric interpretation as above.

\section{On quasicontractive composition semigroups:\\ Remarks and Questions}
\label{Section 3}

As it was mentioned before, if the semigroup $T_t f = f\circ
\phi_t,$ $t\ge 0$, is bounded on a \emph{reflexive} Banach space
$\cB$ satisfying $(\star)$, that is, $\cO(\overline
\bD)\hookrightarrow{} \cB \hookrightarrow{} \Hol(\bD)$ and both
embeddings are continuous, then $\{T_t\}$ is automatically a $C_0$
semigroup. To prove it, one just notices that the functionals
\[
E_z f=f(z), \qquad z\in \mathbb{D}
\]
are complete in $\cB$. Hence condition $(\star)$ implies that the
family $\{T_t\}$ is weakly continuous. By \cite[theorem
5.8]{EngNag-book}, it is strongly continuous.

\medskip

%OLD:
%
%In general, if $\cB$ is not reflexive, not all bounded semigroups of composition operators on $\cB$ are
%$C_0$ semigroups. Moreover, for instance, there is no nontrivial
%$C_0$ semigroups of composition operators
%on $H^\infty$, (see \cite{Sis}, for instance). For these cases, it would be desirable to
%find an analogue of Theorem~\ref{thm-main}, where
%the words ``$C_0$ semigroup'' are substituted by a weaker property, valid for all
%bounded composition semigroups.

In general, if $\cB$ is not reflexive, not all bounded semigroups
of composition operators on $\cB$ are $C_0$ semigroups. Moreover,
for some spaces $\cB$, there is no nontrivial $C_0$ semigroups of
composition operators on $\cB$. See \cite{AJS} for spaces between
$H^\infty$ and the Bloch space, \cite{AreContrRodr16} for certain
mixed norm spaces; in these papers one can find references to
earlier results. For non-reflexive spaces, it would be desirable
to find an analogue of Theorem~\ref{thm-main}, where the words
``$C_0$ semigroup'' are substituted by a weaker property, valid
for all bounded composition semigroups.

\medskip

On the other hand, as Avicou, Chalendar and Partington prove in
\cite{AChP1}, any semigroup of composition operators on
$H^2(\mathbb{D})$ is quasicontractive. Their argument extends to a
wide range of Banach spaces $\cB$. Consider the following condition

\medskip
\begin{enumerate}
\item[$ (\star\star)$] For any univalent function $\eta$, which
maps $\bD$ to $\bD$ and satisfies $\eta(0)=0$, one has
$\|f\circ\eta\|_{\cB}\le\|f\|_{\cB}$.
\end{enumerate}

\medskip
Note that, in particular, $ (\star\star)$ implies that $\mathcal{B}$ is rotation invariant.
Moreover, if we denote
\[
\alpha_{r}(z)= \frac {z+r}{1+r z}\,  ,
\quad r\in (0,1),
\]
it is clear that  $\alpha_{r}$ is a
hyperbolic disc automorphism and the arguments of \cite{AChP1} yield
the following statement.

\begin{proposition}
\label{proposition univ}
Suppose $\cB$ has the properties $(\star)$ and $(\star\star)$.
Then the following holds.

\begin{itemize}
\item[(i)] Every holomorphic flow $\{\phi_t\}$ generates a bounded
semigroup of composition operators on $\cB$ (not necessarily a $C_0$-semigroup), if and only if the
composition operators $C_{\alpha_r}f = f\circ\alpha_r$ are bounded on $\cB$
for any $r\in (0,1)$;

\item[(ii)] Every holomorphic flow $\{\phi_t\}$ generates a
quasicontractive semigroup of composition operators on $\cB$ if
and only if the composition operators $C_{\alpha_r}$ are bounded on $\cB$
for $r\in (0,1)$ and satisfy an estimate
\begin{equation}
\label{Cr-quasic}
\|C_{\alpha_r}\|_{\cL(\cB)}\le \left (
\frac{1+r}{1-r} \right )^{a},
\quad r\in (0,1),
\end{equation}
for some nonnegative constant $a$.  In this case, for any univalent function $\phi:\bD\to \bD$
one has
\begin{equation}
\label{est-univ}
\|C_{\phi}\|_{\cB\to \cB}\le
\left (
\frac{1+|\phi(0)|}{1-|\phi(0)|} \right )^{a}\, .
\end{equation}
\end{itemize}
\end{proposition}

A few words are in order. First, observe that
$\{C_{\alpha_r}\}_{0\leq r<1}$ is a semigroup of operators, if one
makes the change of variables $\displaystyle t=\frac{1}{2}\log
\frac{1+r}{1-r},$ or, equivalently, $r=\tanh t$. Indeed, we have
that $T_t=C_{\alpha_{\tanh t}}$, satisfies $T_t T_s = T_{t+s}$
since $\alpha_r \alpha_s = \alpha_{(r+s)/(1+rs)}$. Inequality
\eqref{Cr-quasic} rewrites as $\|T_t\|_{\cL(\cB)}\le e^{2at}$,
$t\ge 0$, therefore it is equivalent to the fact that $\{T_t\}$ is
a quasicontractive semigroup. By passing to the parameter $t$, it
follows also that if \eqref{Cr-quasic} holds for $r\in(0, r_0)$,
where $0<r_0<1$, then it holds for all $r\in(0, 1)$, and
\eqref{est-univ} is true for any value of $\phi(0)$.

Given a sequence $\beta = \{\beta_n\}$ sequence of positive
numbers, consider the weighted Hardy space $\mathcal{H}^2(\beta)$
consisting of analytic functions $f(z) = \sum_{n=0}^\infty a_n z^n
$ on $\mathbb{D}$ for which the norm
\[
\|f \|_\beta = \bigg( \sum_{n=0}^\infty |a_n|^2  \beta_n^2 \bigg)^{1/2}
\]
is finite.
%
% HE PASADO A OTRO SITIO
%
%We point out that composition operators induced by the M\"obius maps
%$\alpha_r$ are not always bounded even in certain weighted Bergman
%spaces with fast increasing weights (which belong to the
%scale of spaces $\mathcal{H}^2(\beta)$). See, for instance, Chapter 5 in \cite{CMc}.}
%
Consider the quantity
\begin{equation}
\label{sup}
\Lambda:=  \sup\,\bigg \{\Re
\sum_{n=0}^\infty \Big[ (n+1)\beta_n^2 \overline a_n a_{n+1} -n
\beta_{n+1}^2
 a_n \overline a_{n+1}
\Big]
: \, \sum_{n=0}^\infty \beta_n^2 |a_n|^2  = 1 \bigg\}   \;    .
\end{equation}
Gallardo-Guti\'errez and Partington proved in \cite{GallPart} that if
$\cB=\mathcal{H}^2(\beta)$ is a weighted Hardy space which
contains $\mathcal{H}^{2}(\mathbb{D})$, then
$\{C_{\alpha_r}\}$ satisfy \eqref{Cr-quasic} if and only if
$\Lambda<\infty$. Moreover, as they show, the
best constant $a$ in the estimate
\eqref{Cr-quasic} equals to $\Lambda/2$.
(In \cite{GallPart}, Proposition 2.4 was only stated for the case when
$\mathcal{H}^2(\beta)\supset H^2$, but its proof is valid without this assumption.)
%the estimate
%$$
%\|C_{\alpha_r}\|_{\cL(\mathcal{H}^2(\beta))}\leq \left (
%\frac{1+r}{1-r} \right )^{\Lambda/2}
%$$
%holds.
% (equivalently,
%$\|C_{\alpha_{\tanh t}}\|_{\mathcal{H}^2(\beta)}\leq \exp(\Lambda t)$, $t\ge 0$).
Our next observation makes their criterion for quasicontractivity
of $\{C_{\alpha_{\tanh t}}\}$ more explicit.

\begin{proposition}
\label{proposition quasicontractive}
Let $\mathcal{H}^2(\beta)$ be a weighted Hardy space.
% which contains $H^{2}(\mathbb{D})$.
%
%
%Let  $\{C_{\alpha_r}\}_{0\leq r<1}$ be the $C_0$-semigroup of composition
%operators induced by the hyperbolic automorphisms $\alpha_{r}(z)=(z+r)/(1+r z)$.
Then $\{C_{\alpha_{\tanh t}}\}_{t\ge 0}$
is a quasicontractive  $C_0$-semigroup if and only if
\begin{equation}
\label{critsup}
\sup_n\;  n\Big|1-\frac {\beta_{n+1}}{\beta_n} \Big| <\infty.
\end{equation}
\end{proposition}

\begin{proof}[Proof]
Simple calculations show that \eqref{critsup} is equivalent to
\begin{equation}\label{condition quasicontractive}
\sup \Bigl |(n+1)\frac{\beta_n}{\beta_{n+1}} - n \frac{\beta_{n+1}}{\beta_n} \Bigr | <\infty.
\end{equation}

Assume that $\{C_{\alpha_{\tanh t}}\}$ is quasicontractive;
hence  $\Lambda<\infty$, or equivalently
\begin{equation}\label{condition bounded 2}
\sup \Big \{ \Re \sum_{n=0}^\infty
\Big[
(n+1)\frac{\beta_n}{\beta_{n+1}} \,\overline x_n x_{n+1}
-n \frac{\beta_{n+1}}{\beta_n} \,x_n \overline x_{n+1}
\Big]
: \, \sum_{n=0}^\infty |x_n|^2  = 1\Big\}<\infty.
\end{equation}
From here (\ref{condition quasicontractive}) follows for particular choices
of the $\ell^2$-vectors $x=\{x_n\}$ in the unit sphere.

Conversely; let us assume
that (\ref{condition quasicontractive}) holds. Hence (\ref{condition bounded 2}) also holds since clearly
$$
\Re \left (
(n+1)\frac{\beta_n}{\beta_{n+1}} \overline x_n x_{n+1}
-n \frac{\beta_{n+1}}{\beta_n} x_n \overline x_{n+1} \right ) = \left ( (n+1)\frac{\beta_n}{\beta_{n+1}}
-n \frac{\beta_{n+1}}{\beta_n} \right ) \Re (x_n \overline x_{n+1}).
$$
So, $\Lambda$ is finite and therefore,
$$
\|C_{\alpha_r}\|_{\cL(\mathcal{H}^2(\beta))}= \|C_{\alpha_{\tanh t}}\|_{\cL(\mathcal{H}^2(\beta))}
\leq e^{\Lambda t},
$$
which shows that $\{C_{\alpha_r}\}_{0\leq r<1}$ is quasicontractive as we wish.
\end{proof}

The classical Dirichlet space $\cD$ corresponds to the weights,
given by $\beta_n=\root\of n$ for $n\ge 1$ and $\beta_0=1$. This
space satisfies $(\star\star)$ and, in fact, satisfies the
estimate \eqref{est-univ}. (See \cite{Sis-Dir}, \cite{GallMontes}
and \cite{MartVukotic} for estimates for the norms of composition
operators $C_\phi$ on $\cD$.) We get the following statement,
which is close to \cite{GallPart}, Corollary~2.5.

\begin{prop}
Let $\cB=\mathcal{H}^2(\beta)$. Suppose that the sequence $\{\beta_n/\root\of n: n\ge 1\}$
is monotone decreasing, $\beta_0\ge \beta_1$,
and \eqref{critsup} holds. Then any holomorphic flow $\{\phi_t\}$ generates a
quasicontractive semigroup of composition operators on $\cB$
and \eqref{est-univ} holds for any univalent function $\phi:\bD\to\bD$.
\end{prop}

\begin{proof}
Assuming the hypotheses, we get from \cite{Sis-Dir} that
$(\star\star)$ holds for the Dirichlet space.
Next, we apply a result by Cowen \cite[Theorem~7]{Cowen90}, and get
$\|C_\eta\|_{\cL(\cB)}\le \|C_\eta\|_{\cL(\cD)}= 1$
for any univalent function $\eta:\bD\to\bD$ with $\eta(0)=0$. Hence
$(\star\star)$ holds for $\cB$. Now all our statements follow from
Proposition~\ref{proposition univ}.
\end{proof}

We remark that a lemma which implies the cited result by Cowen had
been proved in 1972 by Katznel'son \cite{Katzn} (the proofs are
different). We refer to \cite{ChalPart2014} for more information
and for other applications of this kind of results.

Notice that, whenever $\cB$ is contained in the disc algebra
$A(\bD)$, not all holomorphic flows $\{\phi_t\}$ induce a bounded
semigroup of composition operators on a Banach space $\cB$ This
applies, in particular, to Dirichlet spaces, smaller than $\cD$.
This follows from the observation that $\phi_t\in \cB$ for all
$t>0$ whenever operators $C_{\phi_t}$ are bounded, but there are
flows such that $\phi_t(z)$ does not extend continuously to the
closed unit disc. See also Theorem 4.8 in \cite{CMc}.

There are spaces $\cH^2(\beta)$ where
$C_{\phi}$ is bounded for any univalent $\phi:\bD\to\bD$ but is
unbounded for some non-univalent functions, like
the Dirichlet space (see \cite{FKMR}, for instance).
%
%\textbf{Moreover, there are even weighted Dirichlet spaces such that
%not every univalent self-map of $\mathbb{D}$ induces a bounded
%composition operator: see Theorem 4.8 in \cite{CMc}, for instance.}

On the other side, composition operators induced by the M\"obius
maps $\alpha_r$ are not always bounded even in spaces
$\mathcal{H}^2(\beta)$ with fast decreasing weights. See, for
instance, Chapter 5 in \cite{CMc}, in particular, Theorem~5.2.

%We remark that for any monotone decreasing weight sequence $\{\beta_n\}$, the space
%$\cH^2(\beta)$ satisfies $(\star\star)$  (see \cite{Cowen90}, for instance). This implies that for these weights,
%the estimate \eqref{est-univ} (with $a=\Lambda$) holds for any not necessarily univalent
%holomorphic self-map $\phi:\bD\to\bD$, whenever \eqref{critsup} holds
%(see \cite{GallPart}, Corollary 2.5.).

A key observation regarding Proposition \ref{proposition
quasicontractive} is that the quasicontractivity property of
composition semigroups is very sensitive to changing the norm by
an equivalent one. As it follows from Proposition~\ref{proposition
quasicontractive}, if for some sequence of weights $\{\beta_n\}$,
the semigroup $\{C_{\alpha_{\tanh t}}\}$ is quasicontractive, then
it will fail to be quasicontractive for weights
\[
\tilde \beta_n= (2+(-1)^n)\beta_n,
\]
These weights
define an equivalent norm, so that the property of boundedness of
our semigroup (as well as that of
any other composition semigroup) will not be affected by this change.

This phenomenon is related to much more general facts proved by
Matolcsi in  \cite{Matolcsi_Ban}: given any $C_0$-semigroup
on a Banach space, whose generator is unbounded, it can be converted to a non-quasicontractive
one by passing to
an equivalent norm on this space. By~\cite{Matolcsi_Hilb}, the same is true in the context
of Hilbert spaces.

Observe that, if $\cB$ is a Hilbert space, the generator of any
quasicontractive semigroup on $\cB$ admits an $H^\infty$ calculus
on a half-plane $|\arg(z_0-z)|<\pi/2$, see the book \cite{Haase}.
The existence of an $H^\infty$ calculus is not affected if one
passes to an equivalent norm on $\cB$. This motivates the
following question.

\begin{Question}
Do there exist weights $\{\beta_n\}$ such that the generator of
the semigroup $\{C_{\alpha_{\tanh t}}\}$ on $\cH^2(\beta)$ is
bounded, but does not admit an $H^\infty$ calculus in a sector
$|\arg(z_0-z)|<\theta$, where $\theta\in(0,\pi/2)$? The same can
be asked for an arbitrary bounded composition semigroup on
$\cH^2(\beta)$.
\end{Question}

Finally, we notice that the property $(\star\star)$ and estimates
like~\eqref{est-univ} are known for many classical Banach spaces.
We refer to \cite[Chapter 3]{CMc} for $H^p$ spaces and to
\cite{BouCimaMath} for VMOA. Property $(\star\star)$ also holds
true for the case of mixed norm spaces $H(p,q,\alpha)$,
see~\cite{AreContrRodr16}.

\end{document}